\documentclass{article}
\usepackage[colorlinks=true,urlcolor=blue, linkcolor=blue,pageanchor=false ]{hyperref}
\usepackage{geometry}
\usepackage[english]{babel}
\usepackage{amsmath}
\usepackage{graphicx}
\usepackage{ marvosym }
\usepackage[utf8]{inputenc}
\usepackage{mathtools}
\usepackage{geometry}
\usepackage{amsfonts}
\usepackage{amssymb}
\usepackage{amsthm}
\usepackage{thmtools}
\usepackage{t1enc}
\usepackage[titles]{tocloft}
\usepackage{makeidx}
\usepackage{wasysym}
\usepackage{stmaryrd}
\usepackage{calc}  
\usepackage{enumitem} 
\usepackage[refpage]{nomencl}

\usepackage{algpseudocode}
\usepackage{tikz}
\usetikzlibrary{arrows}
\usepackage{float}

\DeclarePairedDelimiterX\Set[2]{\lbrace}{\rbrace}%
 { #1 \,\delimsize:\, #2 }

\theoremstyle{definition}

\newcommand{\defeq}
{\stackrel{\text{def}}{=}}

\theoremstyle{plain}
\newtheorem{thm}{Theorem}
\newtheorem{prop}[thm]{Proposition}
\newtheorem*{prop*}{Proposition}
\newtheorem*{seged*}{Sublemma}
\newtheorem{cor}[thm]{Corollary}
\newtheorem{lem}[thm]{Lemma}
\newtheorem{conj}[thm]{Conjecture}
\newtheorem*{lem*}{Lemma}
\theoremstyle{definition}

\newtheorem*{defn*}{Definition}

\newtheorem{fel*}[thm]{Exercise}

\newtheorem*{megf*}{Observation}
\theoremstyle{remark}
\newtheorem{rem}[thm]{Remark}
\newtheorem*{rem*}{Remark}

\newenvironment{sbiz}{\par\noindent{\itshape Proof:}\ }{\newmoon}

\newenvironment{nbiz}{\par\noindent{\itshape Proof:}\ }{}

\title{Edmonds' Branching Theorem in Digraphs without Forward-infinite Paths}
\author{Attila Joó\thanks{MTA-ELTE 
Egerváry 
Research Group, 
Department of Operations Research, Eötvös Loránd University, 
Budapest, Hungary. 
E-mail: {\tt 
joapaat@cs.elte.hu}.}}
\date{2014}

\begin{document}

\maketitle
This is the peer reviewed version of the following article: \cite{joo2015edmonds}, which has been published in final form at 
\url{http://dx.doi.org/10.1002/jgt.22001}. This article may be used for non-commercial purposes in accordance with Wiley Terms and 
Conditions for Self-Archiving.

\begin{abstract}
  Let $ D $ be a finite digraph, and let $ V_0,\dots,V_{k-1} $ be 
  nonempty subsets of $ V(D) $. The (strong form of) Edmonds' branching theorem states that
   there are pairwise edge-disjoint spanning branchings $ \mathcal{B}_0,\dots, \mathcal{B}_{k-1} $ in $ D $ such that the 
   root set of $ \mathcal{B}_i $ is $ V_i\ (i=0,\dots,k-1) $ if and only if for all  $ \varnothing \neq X\subseteq V(D) $ the number of 
   ingoing 
   edges 
   of $ X $ is greater than or equal to the number of sets $ V_i $ disjoint from $ X $. 
  As was shown by R. Aharoni and C. Thomassen in \cite{aharoni1989infinite}, 
this theorem does not remain true for infinite 
digraphs. Thomassen also proved that for the class of digraphs without backward-infinite paths, the above theorem of Edmonds  
remains true. Our main result is that for digraphs without forward-infinite paths, Edmonds' branching theorem remains true as 
well.

\end{abstract}
\section{Notions and notation}
The digraphs $ D=(V,A) $ considered here may have multiple edges and arbitrary size. Loops 
are also allowed but are irrelevant to our subject.   
If $ B\subseteq V $, then we write $D[B] $ for the subgraph of $ D $ spanned by $ B $.  
For  $ X \subseteq V $ let  $ \mathsf{in}_{D}(X)$ and $\mathsf{out}_{D}(X) $ be the set of ingoing and  
outgoing edges respectively of $ X $ in $ D $, 
and let $ \varrho_{D}(X),\  \delta_{D}(X) $ be their respective cardinalities. By a path, we mean a directed, possibly 
infinite, simple path  (the
repetition of vertices is not allowed). We denote by $ \mathsf{start}(P) $ and $ \mathsf{end}(P) $  the first and last vertex of 
the 
path 
$ P $, if they exist. For an edge $ e $ from $ x $ to $ y $, let $ \mathsf{start}(e)=x $ and $ \mathsf{end}(e)=y $. For $ X,Y 
\subseteq V $,  let $ 
\mathsf{e}_{D}(X,Y)=\{ e\in A: \mathsf{start}(e)\in X,\ 
\mathsf{end}(e)\in Y \} $; for singletons we write $ \mathsf{e}(x,y) $ instead of $ \mathsf{e}(\{ x \},\{ y \}) $. We say that 
the path $ P $  \textit{goes from} $ X $ \textit{to} $ Y $ if $ V(P)\cap X=\{ \mathsf{start}(P) \} $ 
and $ 
V(P)\cap Y=\{ \mathsf{end}(P) \}\  (\mathsf{start}(P)=\mathsf{end}(P)  \text{ is allowed})$. We 
call  $ \min \{ \varrho_{D}(X): \varnothing\neq X\subseteq V\setminus\{ r \} \} $  the  edge-connectivity of $ D 
$ from $ r $, and  $ D $ is $ \kappa $-edge-connected from $ r $ if this cardinal is at least $ \kappa $. 

 A digraph is an \textit{arborescence} with root vertex $ r $ if it is a directed tree such that all vertices are 
reachable from $ r $. A digraph is a \textit{branching} with root set $ W $ if its weakly connected components are 
arborescences and the 
vertex set $ W $ 
consists of the roots of these arborescences. $ \mathcal{B} $ is a \textit{k-branching} in $ D $ iff it is a $ k 
$-tuple $ 
\mathcal{B}=(\mathcal{B}_0,\mathcal{B}_1,\dots,\mathcal{B}_{k-1}) $ such that the $ \mathcal{B}_i=(V_i,A_i) $'s are 
edge-disjoint 
branchings in $ D $ (not necessarily with the same root sets), and we let $ \mathit{D\bbslash \mathcal{B}}=(V, A 
\setminus\cup_{i< 
 k}A_i) $. If $ F 
 $ is a branching  and $ P $ is a path such that $ V(F)\cap V(P)=\{ \mathsf{start}(P) \} $, then we denote by $ \mathit{F+P} $ 
 the branching $ (V(F)\cup V(P),A(F)\cup A(P)) $.  
\section{Introduction}
Edmonds proved in \cite{edmonds1973edge} his famous theorem (now called the \textit{weak form of Edmods' branching theorem}) which states 
that if a finite  digraph is $ k 
$-edge-connected from a vertex $ r $ for some $ k\in \mathbb{N} $, then it has $ k $ edge-disjoint spanning arborescences rooted 
at $ r $. 
He also proved a 
generalization of this (called the \textit{strong form of Edmods' branching theorem}; see \cite{frank2011connections} p. $ 349 $ Theorem 
$ 10.2.1 $) 
which states the following.  If $ D $ is a finite digraph and $ V_0,\dots,V_{k-1}\subseteq V(D) $, 
  then there are pairwise edge-disjoint spanning branchings $ \mathcal{B}_0,\dots, \mathcal{B}_{k-1} $ in $ D $ such that the root set of 
  $ 
  \mathcal{B}_i $ is $ V_i\ (i=0,\dots,k-1) $ if and only if  all  $ \varnothing \neq X\subseteq V(D) $ has at
  least $ \left|\{i<k: V_i\cap X=\varnothing \}\right| $ ingoing edges. L. Lovász gave a new 
elegant proof for Edmods' branching theorem in \cite{lovasz1976two}, and his techniques opened the door for 
further generalizations 
such as 
\cite{kamiyama2009arc}, \cite{fujishige2010note}, \cite{berczi2009packing}  and \cite{kiraly2013maximal}.
Infinite generalizations have been obstructed by a negative result of R. Aharoni and C. Thomassen \cite{aharoni1989infinite}. 
 They constructed, for any $ k\in \mathbb{N} $,  a countably-infinite, locally finite, simple graph $ G $ such that $ 
G $ has a $ k $-connected orientation but has  vertices $ u,v $ such that deleting the edges of an arbitrary path between $ u 
$ and $ v $ makes the remaining graph disconnected.

Thomassen showed (unpublished) that if $ D=(V,A) $ does not contain backward-infinite paths and  is  $ k $-edge-connected from 
$ r $ for some $ k\in \mathbb{N} $, then it has $ k 
$ edge-disjoint spanning arborescences rooted at $ r $. The main idea of his proof is the following: construct first a 
spanning 
subgraph $ D'=(V,A') $ of $ D $ such that $ D' $ is also $ k $-edge-connected from $ r $ and all vertices of $ D' $ have finite 
indegrees. 
After that, one
can build the desired arborescences in $ D' $ using the finite version of the theorem and compactness arguments. Thomassen's 
proof also works for the strong form of the Edmonds' branching theorem. Our main result is that 
 disallowance of  forward-infinite paths instead of backward-infinite paths is also sufficient. Our proof   
uses techniques very different from  Thomassen's proof.

There is a general approach in finite combinatorics based on separating by ``tight'' sets to smaller subproblems and handling of these 
by induction independently. This approach works for example for Menger's theorem and for Edmonds' branching theorem  but obviously 
can not be used directly  to infinite generalizations because it is possible that the subproblems have the same size as the original. 
Even so we will define the notion of ``tightness''  in the context of Edmonds' branching theorem and it will play key role in our proof. 
An other proof for the finite case 
given by Lovász in \cite{lovasz1976two} makes it possible (even without the restriction about infinite paths) to create edge-disjoint 
branchings with the prescribed root sets where all of them have infinitely many vertices. Unfortunately using Lovász's approach we can 
not guarantee that the resulting branchings will be spanning branching (not even in the countable case) because we can not control that 
which vertex  do we extend a branching with. This controllability will be essential in our proof to ensure conditions after limit steps 
in our recursive construction. 

\section{Main result}
In this section,  we  state and prove our main result. Instead of packing branchings with prescribed root sets, we formulate 
this result
in a formally more general (but in fact equivalent) form, in which we want to extend some initial edge-disjoint branchings to 
edge-disjoint spanning branchings without 
changing their root sets. If these initial branchings have no edges, then we get back the ``prescribed root sets''-approach.
   
\begin{thm}\label{főtétel}
Let  $ D=(V,A) $ be a digraph,   $ k\in \mathbb{N} $  and $ 
\mathcal{B}_i=(V_i,A_i)\ (  i<k) $ 
 edge-disjoint branchings in  $ D $  and let $D\bbslash \mathcal{B}=(V, A \setminus\cup_{i< 
 k}A_i) $. Suppose that $ D\bbslash \mathcal{B} $ does not contain 
 forward-infinite paths. Then the branchings can be 
extended to edge-disjoint spanning branchings of $ D $ without changing their root sets if and only if 

\begin{equation}\label{egy feltetel}
\forall X\  (\varnothing \neq X \subseteq V\Longrightarrow  
 \varrho_{D\bbslash \mathcal{B}}(X)\geq \left|\left\lbrace i<k : V_i\cap X=\varnothing\right\rbrace\right|).
\end{equation}    
\end{thm}
If there is an $ r\in V $ such that $ V_i=\{ r \} $ for all $ i<k $, then we get the following special case.
\begin{cor}
Let the digraph $ D $ be  $ k $-edge-connected from the vertex $ r $ for some $ k\in \mathbb{N} $, and suppose that there are no 
forward-infinite paths in $ D 
$. Then there are $ k $ edge-disjoint spanning arborescences in $ D $ rooted at $ r $.
\end{cor}
\begin{rem}
Our proof of  Theorem \ref{főtétel} also works in a more general case when there is no restriction on the quantity of the 
initial branchings, but all vertices belong to all but finitely many of these branchings.
\end{rem}

\begin{proof}[Proof of Theorem \ref{főtétel}]
 
The necessity of condition (\ref{egy feltetel}) is obvious, so we 
show only that it is sufficient. To do so, we need the following lemma.
\begin{lem}\label{seged}
For any $j<k $ and $ v\in V\setminus V_{j} $, there is a path $ P $ in $ D\bbslash \mathcal{B} $ from $ V_j $ to $ v $  such 
that 
condition 
(\ref{egy feltetel}) holds for $ D $ and  
 the $ k $-branching $ \mathcal{B}' $, where  $ \mathcal{B}_i'=
\begin{cases} \mathcal{B}_i+P &\mbox{if } i=j \\
\mathcal{B}_i & \mbox{otherwise }   
\end{cases} $.   

\end{lem}
Without loss of generality, it is enough to prove Lemma \ref{seged} for $ j=0 $, because the role of the initial branchings are 
symmetric. Before the 
proof, we need to devolp some basic tools 
in the spirit of Lovász's proof for the finite 
version of the theorem in \cite{lovasz1976two}.

\subsection{Basic tools}

We will prove here some facts which are known from finite branching-packing techniques and remain true with the same proof 
in the infinite case. In this subsection, we fix a digraph $ D=(V,A) $ and a $ k $-branching $ \mathcal{B}\   
(\mathcal{B}_i=(V_i,A_i),\ 
i=0,\dots,k-1)$ of $ D $, that satisfy condition (\ref{egy feltetel}). 

Call a set $ \varnothing \neq X \subseteq V 
$ \textbf{tight} (with respect to $ 
\mathcal{B} $), if $ 
\varrho_{D\bbslash \mathcal{B}}(X)=\left|\left\lbrace i<k: V_i\cap X=\varnothing \right\rbrace\right| $; and 
\textbf{dangerous},  
if it 
is tight 
and $ 
X\cap 
V_0\neq \varnothing $. For example, $ V $ itself is dangerous. It is easy to see that if $ e\in 
\mathsf{out}_{D\bbslash \mathcal{B}}(V_0) $,  then the extension $ \mathcal{B}_0'=\mathcal{B}_0+e $  violates condition 
(\ref{egy 
feltetel}) if and only if $ e $ is an ingoing edge of some dangerous set. 

\begin{prop}\label{metszetveszélyes}
If $ X, Y $ are dangerous and $ X\cap Y\neq \varnothing $, then $ X\cap Y $ is also dangerous.
\end{prop}
\begin{sbiz}
Let $ s: \mathcal{P}(V)\rightarrow \mathbb{N},\ s(X)=\left|\left\lbrace i<k: V_i\cap X=\varnothing \right\rbrace\right| $. 
Then $ s $ is supermodular i.e., for  $ X,Y\subseteq V $, we have  $ s(X)+s(Y)\leq s(X\cup Y)+s(X\cap Y) $. Indeed, let $ i<k 
$ be arbitrary. If $ V_i\cap X=\varnothing $ and $ V_i\cap Y=\varnothing $, then $ V_i\cap(X\cup Y)=\varnothing $ and $ V_i\cap 
X\cap Y=\varnothing $, so  $ V_i $'s contribution to both sides of the inequality is $ 2 $. If $ V_i\cap X=\varnothing $ and 
$ 
V_i\cap Y\neq \varnothing $, then  $ V_i $'s contribution to both sides is $ 1 $.
 Observe that equality holds if and only 
if there is no $ V_i $ such that 
 $ V_i\cap X\neq \varnothing,\ V_i\cap Y\neq\varnothing  $ but $ V_i\cap X\cap Y= \varnothing $. Let $ 
p(X)=\varrho_{D\bbslash \mathcal{B}}(X)-s(X) $ (for an infinite cardinal $ \kappa $ and $ n\in \mathbb{N} $ let $ 
\kappa-n=\kappa 
$). Then  condition 
(\ref{egy 
feltetel}) is equivalent with the requirement $  p(X)\geq 0 $ for all $ X\neq \varnothing $, and 
the tightness of $ X $ means $ p(X)=0 $. The function  $ \varrho_{D\bbslash \mathcal{B}} $ is submodular, therefore so is $ p $ i.e. $ 
p(X)+p(Y)\geq p(X\cup Y)+p(X\cap Y) $ holds for all $ X,Y\subseteq V $. Let $ X, Y $ be dangerous, 
and $ 
 X\cap 
 Y\neq \varnothing $. Then by submodularity and by condition (\ref{egy 
feltetel}), we get

\[ 0+0=p(X)+p(Y)\geq p(X\cup Y)+p(X\cap Y)\geq 0+0, \]

\noindent so $ X\cup Y $ and $ X\cap Y $ are tight. Therefore  $s(X)+s(Y)=s(X\cup Y)+s(X\cap Y) $. By the observation 
about the 
function $ s $, we may conclude from $ X\cap V_0\neq 
\varnothing,\ Y\cap V_0\neq \varnothing  $ that  $ X\cap Y\cap V_0 \neq 
\varnothing 
$, so $ X\cap Y $ is dangerous. 
\end{sbiz}

\begin{prop}\label{elérhető}
Let $ B $ be a dangerous set. Then for any $ w\in B $, there is a path $ R $ from $ V_0\cap B $ to $ w $ in $ 
(D\bbslash \mathcal{B})[B] 
$.
\end{prop}
\begin{sbiz}
Let $ B' $ be the set of vertices which are reachable from $ V_0\cap B $ in  $ (D\bbslash \mathcal{B})[B] $. Suppose, that $ 
B'\neq B $. Then $ B\setminus B' $ violates condition (\ref{egy 
feltetel}), which is a contradiction.
\end{sbiz}

\begin{prop}\label{path condition}
For all $ w\in V $, there is a system of edge-disjoint paths $ 
\left\lbrace P_{i} \right\rbrace_{i<k} $ in $ D\bbslash \mathcal{B} $ such that $P_i $ goes from $ V_i $ to $ w $. 
\end{prop}
\begin{sbiz}
 We extend $ D\bbslash \mathcal{B} $ to $ H $ by adding new vertices and edges (see figure \ref{ellenpeldakep}). Let $ 
 V(H)=V\cup 
 \left\lbrace 
 s 
\right\rbrace\cup \left\lbrace v_i 
\right\rbrace_{i<k},\  |\mathsf{e}_{H}(s,v_i)|=1\ (i<k) $ and  $  
|\mathsf{e}_{H}(v_i,u)|=\aleph_0\ (i<k, u\in V_i) $. If there are $ k $ 
edge-disjoint paths from $ s $ to $ w $ in $ H $, then we are done. Suppose, seeking a contradiction, that there are not. By 
Menger's theorem, there 
is a $w\in X\subseteq V(H)\setminus\left\lbrace s \right\rbrace $ with $ \varrho_{H}(X)<k $. Let $ 
l=\left|\left\lbrace v_i \right\rbrace_{i<k}\setminus X\right| $. Note that $ 0<l $, otherwise $ sv_i\in \mathsf{in}_{H}(X)\ 
(i<k) $ and hence $ 
k\leq\varrho_{H}(X) $ would follow. Since there are infinitely many parallel 
edges, $ 
X\cap V $ is disjoint from at least 
 $ l $ branchings. Otherwise $ \varrho_{H}(X)=\varrho_{D\bbslash\mathcal{B}}(X\cap V)+(k-l) $, so $ 
 \varrho_{D\bbslash\mathcal{B}}(X\cap 
V)=l+(\varrho_{H}(X)-k)<l  $, but then $ X\cap V $ violates condition (\ref{egy 
feltetel}) in $ D $ giving us a contradiction.
\end{sbiz}

\begin{figure}[h!]
\centering

\begin{tikzpicture}

\draw  (-5,3) node (v11) {} ellipse (2 and 1.3);
\draw  (-3,1) node (v17) {} ellipse (2 and 1.5);
\draw  (-11,0) node (v5) {} ellipse (1 and 1);
\node [outer sep=0,inner sep=0,minimum size=0] (v1) at (-10,2) {$\Huge{s}$};
\node [outer sep=0,inner sep=0,minimum size=0] (v2) at (-10,1) {$v_0$};
\node [outer sep=0,inner sep=0,minimum size=0] (v3) at (-8,2) {$v_1$};
\node [outer sep=0,inner sep=0,minimum size=0] (v4) at (-8,1) {$v_2$};

\draw  (v1) edge[->] (v2);
\draw  (v1) edge[->] (v3);
\draw  (v1) edge[->] (v4);

\node [outer sep=0,inner sep=0,minimum size=0] (v6) at (-10.5,-0.5) {};
\node [outer sep=0,inner sep=0,minimum size=0] (v7) at (-11.5,0) {};
\node [outer sep=0,inner sep=0,minimum size=0] (v8) at (-11,-0.5) {};
\node [outer sep=0,inner sep=0,minimum size=0] (v10) at (-5,4) {};
\node [outer sep=0,inner sep=0,minimum size=0] (v13) at (-4,2) {};
\node [outer sep=0,inner sep=0,minimum size=0] (v12) at (-2,1) {};
\node [outer sep=0,inner sep=0,minimum size=0] (v14) at (-3,2) {};
\node [outer sep=0,inner sep=0,minimum size=0] (v16) at (-3,0) {};
\node [outer sep=0,inner sep=0,minimum size=0] (v15) at (-2,0) {};
\node [outer sep=0,inner sep=0,minimum size=0] (v18) at (-4,1.5) {};
\node (v9) at (-4,3.5) {};
\node (v19) at (-4.5,2.5) {};

\draw  (v2) edge[->, line width=2] (v5);
\draw  (v2) edge[->, line width=2] (v6);
\draw  (v2) edge[->, line width=2] (v7);
\draw  (v2) edge[->, line width=2] (v8);

\draw  (v3) edge[->, line width=2] (v10);
\draw  (v3) edge[->, line width=2] (v11);

\draw  (v3) edge[->, line width=2] (v13);
\draw  (v4) edge[->, line width=2] (v12);
\draw  (v4) edge[->, line width=2] (v13);
\draw  (v4) edge[->, line width=2] (v14);
\draw  (v4) edge[->, line width=2] (v15);
\draw  (v4) edge[->, line width=2] (v16);
\draw  (v4) edge[->, line width=2] (v17);
\draw  (v4) edge[->, line width=2] (v18);
\draw  (v3) edge[->, line width=2] (v9);
\draw  (v3) edge[->, line width=2] (v19);

\draw  (-3.5,-1) node (v20) {} rectangle (-8.5,3.2);

\node at (-11,1.2) {$V_0$};
\node at (-2,2.5) {$V_2$};
\node at (-5,4.5) {$V_1$};
\node at (-8,3.4) {$X$};

\end{tikzpicture}
\caption{The construction of $ H $ and the cut $ X $ from the proof above in the case $ k=3,\ l=1 $.  (Thick 
arrows stand for countably infinite parallel edges).} 
\label{ellenpeldakep}
\end{figure}
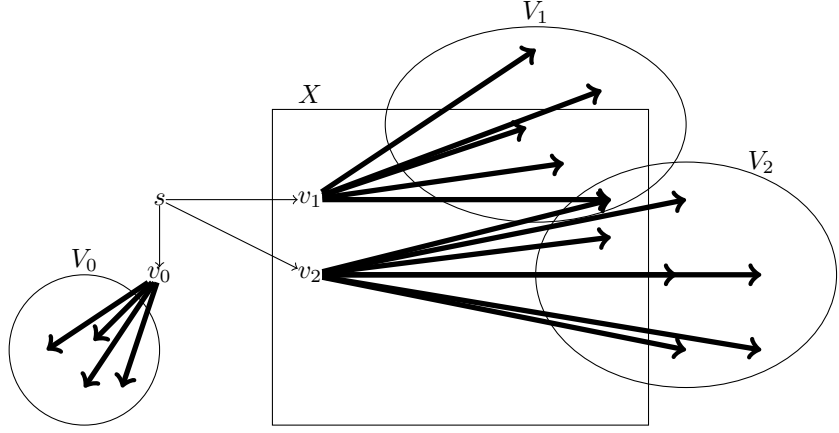

\begin{cor}\label{végtelenutak}
Let $ B_1 \subseteq B_0 $ be dangerous sets and let $\varrho_{D\bbslash \mathcal{B}}(B_0)=\varrho_{D\bbslash \mathcal{B}}(B_0)=l\geq 
1 $. Let $ s_j=\mathsf{end}(e_j) $, 
where $ \{ e_1,\dots,e_l \}=\mathsf{in}_{D\bbslash \mathcal{B}}(B_0) $. 
Then there is a system of edge-disjoint paths $ \left\lbrace P_{j} \right\rbrace_{j=1}^{l} $  in $ 
(D\bbslash \mathcal{B})[B_0] $ such that $ 
P_j$ goes from $ s_j $ to $ B_1 $. Such a path system necessarily contains all of the elements of $ 
\mathsf{e}_{D\bbslash \mathcal{B}}(B_0,B_1\setminus B_0) $, and the multiset of the endpoints of their elements is  $ \{\mathsf{end}(e): 
e \in 
\mathsf{in}_{D\bbslash \mathcal{B}}(B_1) 
\} $.
\end{cor}

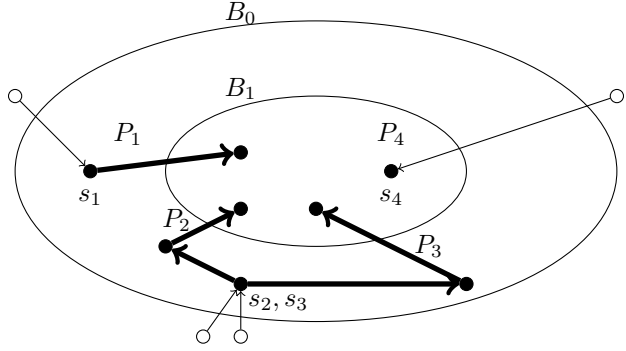
\begin{figure}[H]
\centering

\begin{tikzpicture}

\draw  (0,0) node  {} ellipse (2 and 1);
\draw  (0,0) node  {} ellipse (4 and 2);

\node at (-1,1.1) {$B_1$};
\node at (-1,2.1) {$B_0$};
\node at (1,0.5) {$ P_4 $};
\node at (-2.5,0.5) {$ P_1 $};
\node at (1,-0.35) {$ s_4 $};
\node at (-0.5,-1.75) {$ s_2,s_3 $};
\node at (-3,-0.35) {$ s_1 $};
\node at (1.5,-1) {$ P_3 $};
\node at (-1.85,-0.65) {$ P_2 $};

\node[circle, fill=black,inner sep=0pt,draw,minimum size=5] (v1) at (-3,0) {};
\node[circle, fill=black,inner sep=0pt,draw,minimum size=5] (v2) at (-1,-1.5) {};
\node[circle, fill=black,inner sep=0pt,draw,minimum size=5] (v3) at (-1,0.25) {};
\node[circle, fill=black,inner sep=0pt,draw,minimum size=5] (v4) at (1,0) {};
\node[circle, fill=black,inner sep=0pt,draw,minimum size=5] (v6) at (0,-0.5) {};
\node[circle,inner sep=0pt,draw,minimum size=5] (v7) at (-1,-2.2) {};
\node[circle, fill=black,inner sep=0pt,draw,minimum size=5] (v8) at (-1,-0.5) {};
\node[circle,inner sep=0pt,draw,minimum size=5] (v9) at (4,1) {};
\node[circle,inner sep=0pt,draw,minimum size=5] (v10) at (-4,1) {};
\node[circle,inner sep=0pt,draw,minimum size=5] (v11) at (-1.5,-2.2) {};
\node[circle, fill=black, inner sep=0pt,draw,minimum size=5] (v12) at (-2,-1) {};
\node[circle, fill=black,inner sep=0pt,draw,minimum size=5] (v13) at (2,-1.5) {};

\draw  (v9) edge[->] (v4);
\draw  (v10) edge[->] (v1);
\draw  (v1) edge[->,line width=2] (v3);
\draw  (v11) edge[->] (v2);
\draw  (v7) edge[->] (v2);
\draw  (v2) edge[->,line width=2] (v12);
\draw  (v12) edge[->,line width=2] (v8);
\draw  (v2) edge[->,line width=2] (v13);
\draw  (v13) edge[->,line width=2] (v6);

\end{tikzpicture}
\caption{Corollary \ref{végtelenutak} in the case $ l=4 $. We thickened the desired path system $ \{ P_j \}_{j=1}^{4} $. In 
this example, $ s_2=s_3 $ and $ P_4 $ consists of the vertex $ s_4 $. } 
\label{kétveszélyesfigura}
\end{figure}

\begin{sbiz}
$ B_0 $ is disjoint from exactly $ l $ many  of the sets $ V_i $  because $ B_0 $ is dangerous and $ \varrho_{D\bbslash 
\mathcal{B}}(B_0)=l 
$. Without loss of generality, we may assume that these sets are $ V_1,V_2,\dots,V_l $. By Proposition \ref{path condition}, 
there 
is a system 
of edge-disjoint paths $ \{ P_j' \} _{j=1}^{l}$ in $ D\bbslash \mathcal{B} $ such that $ P_j' $ goes from $ V_j $ to $ B_1 $. 
Note that such a 
path system necessarily contains all the edges in $ \mathsf{in}_{D\bbslash \mathcal{B}}(B_0)\cup \mathsf{in}_{D\bbslash 
\mathcal{B}}(B_1) $, and all the paths enter 
 $ B_0 $ exactly once. By deleting the initial segments of the  paths $ P_j' $ that are not in $ B_0 $, we get the desired 
path system.
\end{sbiz}

\subsection{Proof of the main Lemma}

Now we are able to to prove Lemma \ref{seged}.
\begin{nbiz}
Assume, seeking a contradiction, that Lemma \ref{seged} is false and $ v\in V\setminus V_0 $ witnesses this.
We will construct three sequences: $ \mathcal{B}_0^{n}=(V_0^{n},A_0^{n}), B_n, e_{n} \ (n\in \mathbb{N}) $. Let $ B_0=V,\ 
\mathcal{B}_0^{0}=\mathcal{B}_0 $ and let $ e_0 $ be an arbitrary edge. We will denote the $ k $-branching $ 
(\mathcal{B}_0^{n},\mathcal{B}_1,\dots,\mathcal{B}_{k-1}) $ by $ \mathcal{B}^{n} $.

\noindent Let $ Q $ be a path from 
$ V_0 $ to $ v $ in $ D\bbslash \mathcal{B} $ (such a path exists by Proposition \ref{elérhető}). Let $ u $ be the last vertex 
of $ Q $ 
for 
which there is 
a path  $ R $ from $ V_0 $ to $ u $ in $  D\bbslash \mathcal{B} $ such that   $ \mathcal{B}_0^{1}\defeq\mathcal{B}_0+R $ does 
not 
violate condition (\ref{egy feltetel}). Since $ u $ cannot be the last vertex of $ Q $, there is a unique outgoing edge $ e_1 $ 
of $ u $ 
which is in $ Q $ (see figure \ref{Bn-eképítésefigura}). The extension $ 
\mathcal{B}_0^{1}+e_1 $ violates condition (\ref{egy feltetel}) because of the choice of $ u $, and thus $ e_1\in 
\mathsf{in}_{D\bbslash\mathcal{B}^{1}}(B_1) $ where $ B_1 $ is a set which is dangerous with respect to $ \mathcal{B}^{1} $.

 Our 
plan is to continue by doing the same  but inside $ B_1 $.  Let $ 
Q_1 $ be an arbitrary path 
from $ V_0^{1}\cap B_1 $ to $ \mathsf{end}(e_1) $ in $ (D\bbslash\mathcal{B}^{1})[B_1] $ (such a path exists by Proposition 
\ref{elérhető}). 
Let $ u_1 $ 
be 
the last vertex of $ Q_1 $ 
for which there is 
a path  $ R_1 $ in $  (D\bbslash\mathcal{B}^{1})[B_1] $ from $ V_0^{1}\cap B_1 $ to $ u_1 $ such that   $ 
\mathcal{B}_0^{2}\defeq\mathcal{B}_0^{1}+R_1 $ does not violate condition (\ref{egy feltetel}). Since $ u_1\neq 
\mathsf{end}(e_1) $, there 
is a unique outgoing edge $ e_2 $ of $ u_1 $ which is in $ Q_1 $. The extension $ 
\mathcal{B}_0^{2}+e_2 $ violates condition (\ref{egy feltetel}) because of the choice of $ u_1 $, thus $ e_2\in 
\mathsf{in}_{D\bbslash\mathcal{B}^{2}}(B_2) $, where $ B_2\subsetneq B_1 $ is a set which is dangerous with respect to $ 
\mathcal{B}^{2} 
$ (if $ B_2 \not\subseteq B_1 $, then by Proposition \ref{metszetveszélyes}, we may replace $ B_2 $ with $ B_2\cap B_1 $).

\begin{figure}[H]
\centering

\begin{tikzpicture}

\node at (-4,1.7) {$V_0$};
\node at (3,2.2) {$B_1$};
\node at (3.4,1.7) {$B_2$};
\node  at (-0.5,1.5) {$Q$};
\node  at (-1.5,0) {$R$};
\node  at (2.5,-1.5) {$ R_1 $};
\node  at (3,0) {$ Q_1 $};
\node at (0.3,1.1) {$e_1$};
\node at (3.7,-2) {$e_2$};

\node[circle,inner sep=0pt,draw,minimum size=5] (v1) at (-3,1) {};
\node[circle,inner sep=0pt,draw,minimum size=5] (v2) at (-2,1) {};
\node[circle,inner sep=0pt,draw,minimum size=5] (v3) at (-1,1) {};
\node[circle,inner sep=0pt,draw,minimum size=5] (v4) at (0,1) {$u$};
\node[circle,inner sep=0pt,draw,minimum size=5] (v5) at (1,1) {};
\node[circle,inner sep=0pt,draw,minimum size=5] (v6) at (-4,0) {};
\node[circle,inner sep=0pt,draw,minimum size=5] (v7) at (2,1.7) {$v$};
\node[circle,inner sep=0pt,draw,minimum size=5] (v9) at (0.7,-3) {};
\node[circle,inner sep=0pt,draw,minimum size=5] (v10) at (4,-3) {$u_1$};
\node[circle,inner sep=0pt,draw,minimum size=5] (v11) at (3,-1) {};
\node[circle,inner sep=0pt,draw,minimum size=5] (v8) at (1,-2) {};
\node[circle,inner sep=0pt,draw,minimum size=5] (v12) at (2.7,-2) {};
\node  at (-0.5,1.5) {$Q$};

\draw  (-5,1.5) rectangle (-2.7, -2);
\draw  (0.5,2) rectangle (5,-3.3);
\draw  (1.3,1.5) rectangle (4,-2.5);

\draw  (v3) edge[->,line width=2] (v8);
\draw  (v8) edge[->,line width=2] (v4);
\draw  (v1) edge[->] (v2);
\draw  (v2) edge[->,line width=2] (v3);
\draw  (v3) edge[->] (v4);
\draw  (v4) edge[->] (v5);
\draw  (v6) edge[->,line width=2] (v2);
\draw  (v5) edge[->] (v7);
\draw  (v8) edge[->,dashed] (v9);
\draw  (v9) edge[->,dashed] (v10);
\draw  (v10) edge[->,dashed] (v11);
\draw  (v11) edge[->,dashed](v5);

\draw  (v8) edge[->,line width=3](v12);
\draw  (v12) edge[->,line width=3](v10);

\end{tikzpicture}
\caption{The process described above. The path $ Q $ is represented with a normal line, $ R $ with a thick line, $ Q_1 $ with a 
dashed line and $ R_1 $ with a very thick line. } 
\label{Bn-eképítésefigura}
\end{figure}
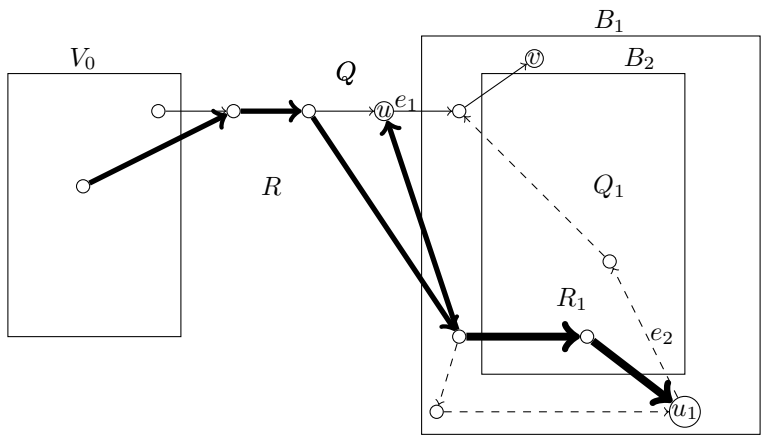

By 
continuing the process  recursively we get the desired sequences with the following properties:

 for all $ n\in \mathbb{N} $:
\begin{enumerate}
\item  $  B_{n+1}\subsetneq B_n $,
\item \begin{enumerate}
\item $ \mathcal{B}^{n}  $ satisfies condition (\ref{egy feltetel}),
\item the sets $ B_0,\dots, B_{n} $ are dangerous with respect to  $ \mathcal{B}^{n}$,
\item $  \mathsf{in}_{D\bbslash\mathcal{B}^{n}}(B_{n})=\mathsf{in}_{D\bbslash\mathcal{B}^{n+1}}(B_{n})  $,
\end{enumerate}     
\item $ e_{n+1}\in \mathsf{e}_{D\bbslash\mathcal{B}^{n+1}}(B_n\setminus B_{n+1},B_{n+1})  $, (and so the  edges $ 
e_{n+1} \ (n\in \mathbb{N}) $ are pairwise distinct).
\end{enumerate}

By throwing away the first finitely many elements of the sequences constructed above and reindexing them, we may assume that all 
the 
members 
of the monotone decreasing sequence $ B_n $ are disjoint from exactly the same, say $ l $ many, of sets among $ V_1,\dots, 
V_{k-1} $. 
Without loss of 
generality we may assume that these sets are $ V_1,V_2,\dots, V_l $. Note that $ l\geq 1 $ because $  B_n $ 
is dangerous with respect to $ \mathcal{B}^{n} $ and $ \varrho_{D\bbslash\mathcal{B}^{n}}(B_n)\geq 1 $ because $ e_n\in 
\mathsf{in}_{D\bbslash\mathcal{B}^{n}}(B_n) $.

 For $ n\in \mathbb{N} $, let  $ \{ P_j^{n} \}_{j=1}^{l} $ be a system of obtained by applying Corollary 
\ref{végtelenutak} with  $  B_{n+1} \subsetneq B_n $ and  $ \mathcal{B}^{n+1} $. Note that $ e_{n+1}\in 
\mathsf{e}_{D\bbslash\mathcal{B}^{n+1}}(B_n\setminus B_{n+1},B_{n+1})\subseteq\bigcup_{j=1}^{l}A(P_j^{n}) $. The multisets $ \{ 
\mathsf{end}(P_j^{n}) \}_{j=1}^{l} $ and $ \{ \mathsf{start}(P_j^{n+1}) \}_{j=1}^{l} $ are equal (they are  $ \{ 
\mathsf{end}(e): e \in\mathsf{in}_{D\bbslash\mathcal{B}^{n}}(B_{n})=\mathsf{in}_{D\bbslash\mathcal{B}^{n+1}}(B_{n}) \} $), so 
we can 
concatenate the 
path systems $ \{ P_j^{n} \}_{j=1}^{l} $  and  $ \{ P_j^{n+1} \}_{j=1}^{l} $ for all $ n $. Thus, we obtain a
system of edge-disjoint paths $ \{ P_j \}_{j=1}^{l} $ in $ D\bbslash \mathcal{B} $ (see figure \ref{vegtelenútfigura} ) 
such that $ \{ e_n \}_{n=1}^{\infty}\subseteq \bigcup_{j=1}^{l}A(P_j) $, and therefore at least one of them is forward-infinite, 
which contradicts  the conditions of Theorem \ref{főtétel}. \rule{1.5ex}{1.5ex} \end{nbiz}

\begin{figure}[H]\label{végtelenútakfigura}
\centering

\begin{tikzpicture}

\draw (-6,-2) -- (6,-2) -- (6,2) -- (-6,2) -- (-6,-2);
\draw (-4,-1.5) -- (5.5,-1.5) -- (5.5,1.5) -- (-4,1.5) -- (-4,-1.5);
\draw (-2,-1) -- (5,-1) -- (5,1) -- (-2,1) -- (-2,-1);
\draw (0,-0.5) -- (4.5,-0.5) -- (4.5,0.5) -- (0,0.5) -- (0,-0.5);
\draw (1.5,-0.3) -- (4,-0.3) -- (4,0.3) -- (1.5,0.3) -- (1.5,-0.3);

\node at (-3.7,-0.2) {$e_1$};
\node at (-1,0.6) {$e_2$};
\node at (-0.6,-0.4) {$e_3$};
\node at (1,0.2) {$e_4$};

\node at (-3.7,1) {$P_1$};
\node at (-4.2,-1.5) {$P_3$};
\node at (-4.2,-0.2) {$P_2$};

\node at (-6,2.2) {$B_0$};
\node at (-4,1.7) {$B_1$};
\node at (-2,1.2) {$B_2$};
\node at (0,0.7) {$B_3$};

\node[circle,inner sep=0pt,draw,minimum size=5] (v1) at (-7,1.5) {};
\node[circle,inner sep=0pt,draw,minimum size=5] (v2) at (-7,-0.5) {};
\node[circle,inner sep=0pt,draw,minimum size=5] (v3) at (-3,1) {};
\node[circle,inner sep=0pt,draw,minimum size=5] (v4) at (-3,0) {};
\node[circle,inner sep=0pt,draw,minimum size=5] (v5) at (-3,-1) {};
\node[circle,inner sep=0pt,draw,minimum size=5] (v6) at (-5,0) {};
\node[circle,inner sep=0pt,draw,minimum size=5] (v7) at (-5,-1) {};
\node[circle,inner sep=0pt,draw,minimum size=5] (v8) at (-1,0) {};
\node[circle,inner sep=0pt,draw,minimum size=5] (v9) at (-1,-0.7) {};
\node[circle,inner sep=0pt,draw,minimum size=5] (v10) at (0.5,0) {};
\node (v11) at (2,0.2) {};
\node (v12) at (2,-0.2) {};

\draw  (v2) edge[->,dotted] (v7);
\draw  (v2) edge[->,dotted] (v6);
\draw  (v1) edge[->,dotted] (v6);
\draw  (v6) edge[->] (v3);
\draw  (v6) edge[->] (v4);
\draw  (v7) edge[->] (v5);
\draw  (v4) edge[->] (v8);
\draw  (v5) edge[->] (v9);
\draw  (v3) edge[->] (v10);
\draw  (v8) edge[->] (v10);
\draw  (v10) edge[->] (v11);
\draw  (v10) edge[->] (v12);
\draw  (v9) edge[->] (v12);

\end{tikzpicture}
\caption{The (initial segment of) path system $ \{ P_j \}_{j=1}^{l} $ in the case $ l=3 $. } 
\label{vegtelenútfigura}
\end{figure}
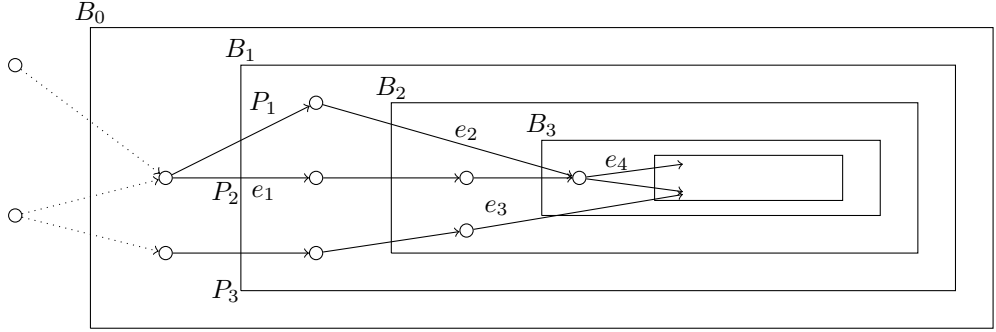

\subsection{Proof of the Theorem }

Now, we continue the proof of Theorem \ref{főtétel}. If $ v\in V $, then by Lemma \ref{seged}, we can 
extend the branchings,  without  violating condition (\ref{egy 
feltetel}), with finitely 
many new vertices and edges such that all of these 
extensions contain $ v $. In the countable case, we can construct the 
desired spanning branchings by the following recursion. In the $ n $-th step, do the extensions above with the 
branchings after the 
previous 
step and with the next vertex $ v_n $ where $ V=\left\lbrace v_n \right\rbrace_{n=0}^{\infty} $. In the uncountable case, we 
have 
to be more careful because we can not avoid limit steps, and we
need to assure that we do not violate condition (\ref{egy 
feltetel}) in these steps as well. The easy trick to handle this is that if we extend in one step one of the branchings with some 
vertex $ v $, then before the next limit step  we put $ v $ into all the branchings which missed it. 
 
Let us make this precise. Let $ V\defeq\left\lbrace 
v_{\alpha} 
\right\rbrace_{\alpha<\lambda} $, where $ \lambda=\left|V\right| $. We extend the branchings by transfinite recursion on $ 
\lambda $.  Denote by $ \mathcal{B}_i^{\alpha}=(V_i^{\alpha},A_i^{\alpha}) $   the branching which we get from $ 
\mathcal{B}_i $  after  the $ \alpha $-th step,  and let $ 
\mathcal{B}^{\alpha}=(\mathcal{B}_0^{\alpha},\mathcal{B}_1^{\alpha},\dots,\mathcal{B}_{k-1}^{\alpha}) 
$ for $   \alpha\leq \lambda $.\\

\noindent Let $ \mathcal{B}_i^{0}=\mathcal{B}_i\ (i<k) $. 

\noindent If $ \alpha<\lambda $ is a limit ordinal, then let $ \mathcal{B}_i^{\alpha}=
(\bigcup_{\beta<\alpha}V_i^{\beta},\bigcup_{\beta<\alpha}A_i^{\beta})$.

\noindent If $ \alpha=\beta+1 $ where $ \beta<\lambda $ is a limit ordinal and  $\mathcal{B}^{\beta} $ satisfies condition 
(\ref{egy 
feltetel}), then  add $ v_\beta $ to all of the  branchings $ \{ \mathcal{B}_i^{\beta}  \}_{i<k} $   by using Lemma \ref{seged} 
repeatedly. Denote the resulting $ k $-branching by $ \mathcal{B}^{\alpha} $.

\noindent If $ \alpha=\beta+2 $ where $ \beta<\lambda $ is an arbitrary ordinal and   $ \mathcal{B}^{\beta+1}$ 
 satisfies condition (\ref{egy 
feltetel}) and the set  $N^{\beta}\defeq\left\lbrace v_{\beta+1 }
\right\rbrace\cup\bigcup_{i<k}V_i^{\beta+1}\setminus 
V_i^{\beta}  $ is finite, then add the elements of $ N^{\beta} $ one by one to all of the branchings $ \{ 
\mathcal{B}_i^{\beta+1} 
\}_{i<k} $   by using Lemma \ref{seged} repeatedly. Denote the resulting $ k $-branching   $ 
\mathcal{B}^{\alpha} $.  

\begin{prop}\label{limeszlépés}
The transfinite recursion above does not stop before the $ \lambda $-th step. 
\end{prop}
\begin{sbiz}
Suppose, seeking a contradiction, that it does. The limit steps are well defined. At successor steps, we do not violate 
condition 
(\ref{egy 
feltetel}), and we have extended the branchings 
with only 
finitely many new vertices and edges. Thus if $ \mathcal{B}^{\gamma+1} $ is well defined for some $ \gamma<\lambda $, then so is 
$ \mathcal{B}^{\gamma+2} $. Hence the first step  
where the recursion can not be continued is necessarily a successor of a 
limit ordinal $ \beta $. 
But  then  $ \beta $ is the first ordinal such that  $ \mathcal{B}^{\beta} $  violates 
condition (\ref{egy 
feltetel}).
Consider the function $ s_\beta(X)=| \{ i<k: V_i^{\beta}\cap X=\varnothing \}|
$, and fix an arbitrary $ \varnothing \neq X\subseteq V $. If $ 
\mathsf{in}_{D\bbslash \mathcal{B}^{\beta}}(X)=\mathsf{in}_{D\bbslash\mathcal{B}}(X) $, then 
$\varrho_{D\bbslash\mathcal{B}^{\beta}}(X)= 
\varrho_{D\bbslash \mathcal{B}}(X) \geq 
s_0(Y)\geq s_{\beta}(X)  $. Otherwise there exists an $ e\in \mathsf{in}_{D\bbslash\mathcal{B}}(X) \setminus  
\mathsf{in}_{D\bbslash\mathcal{B}^{\beta}}(X) $, and so there is 
an $ i_0<k $ 
such 
that $ e\in A_{i_0}^{\beta} $. Let $\gamma<\beta $ be the smallest ordinal such that $ e\in A_{i_0}^{\gamma} $. Then $ \gamma $ 
is  a 
successor ordinal, so by the recursion we have $ \mathsf{end}(e)\in V_i^{\gamma+1}\ (i<k) $. Thus  $ 
\mathsf{end}(e)\in 
V_i^{\beta}\ (i<k) $, and therefore $ \varrho_{D\bbslash\mathcal{B}^{\beta}}(Y)\geq 0=s_{\beta}(X) $. Hence no set $ X $ 
violates condition (\ref{egy feltetel}) with respect to
 $ \mathcal{B}^{\beta} $,  which is a contradiction. 
\end{sbiz}\\

\noindent Finally,   $  \mathcal{B}_i^{\lambda}\ (i=0,\dots,k-1) $  are the desired 
spanning branchings.
\end{proof}

\section{A conjecture about packing infinitely many branchings}
\noindent We show, by a counterexample, that the finiteness of the number of initial branchings is necessary in Theorem 
\ref{főtétel}, and formulate a conjecture with a very natural condition about paths (which is a strengthening of the 
condition  (\ref{egy feltetel})) motivated by this counterexample.
  
We construct a digraph $ D $  and a system of edge-disjoint branchings $ \mathcal{B}=\left\langle \mathcal{B}_n: n\in 
\mathbb{N}   
\right\rangle $  of $ D $  such that  $ D\bbslash \mathcal{B} 
$ 
does not 
contain infinite paths 
and the system satisfies condition (\ref{egy 
feltetel}), but the desired extensions of the branchings do not exist. Let $ D=(V,A) $ where $ 
V=\{ r_n 
\}_{n\in 
\mathbb{N}}\cup \{ v \},\ \left| \mathsf{e}_D(r_0,r_n)\right|=\aleph_0\ (n\in \mathbb{N}^{+}),\ \left| 
\mathsf{e}_D(r_n,v)\right|=1\ 
(n\in 
\mathbb{N}^{+}),\ \left| \mathsf{e}_{D}(v,r_0)\right|=\aleph_0 $ and $ \mathcal{B}_n=(\{ r_n \},\varnothing)\ (n\in \mathbb{N}) 
$ (see figure \ref{ellenpelda}).  We show that the system above does not violate condition (\ref{egy 
feltetel}). Let 
$ \varnothing \neq X \subseteq V $ be arbitrary. Assume first that $ r_0\notin X $. If $ r_n\in X $, for some $ n\in 
\mathbb{N}^{+} 
$, 
then 
$ \varrho_{D\bbslash \mathcal{B}}(X)=\aleph_0 $; if not, then $ X=\{ v \}  $ and $ \varrho_{D\bbslash \mathcal{B}}(X)=\aleph_0 $ 
again. 
Assume $ r_0\in X $. If $ 
v\notin X $, then $ \varrho_{D\bbslash \mathcal{B}}(X)=\aleph_0 $. If $ v\in X $, then
    each element of $ \{ r_n \}_{n\in \mathbb{N}^{+}}\setminus X $ has one outgoing edge to $ X $, then there is  
equality in condition (\ref{egy 
feltetel}). 
 Otherwise, there obviously is no system of 
edge-disjoint paths $ \{ P_n \}_{n\in \mathbb{N}} $ such that $ P_n $ goes from $ r_n $ to $ v $, and thus we can not 
extend  the branchings $ \mathcal{B}_n $ to edge-disjoint spanning branchings.\\

\begin{figure}[H]
\centering
\begin{tikzpicture}
 \tikzstyle{green_circle} = [  font={\huge\bfseries}, shape=circle, minimum size=0.5cm, circular drop shadow, text=black, 
      very thick, draw=black!55, top color=white,bottom color=green!80, text width=0.5cm, align=center]
    
      \tikzstyle{green_circle} = [  font={\huge\bfseries}, shape=circle, minimum size=0.5cm, circular drop shadow, text=black, 
      very thick, draw=black!55, top color=white,bottom color=green!80, text width=0.5cm, align=center]
    
\node (v8) at (-4,-1.2) {$v$};
\node (v5) at (-8,-1) {$r_3$};
\node (v1) at (-12,-1.2) {$r_0$};
\node (v4) at (-8,0) {$r_2$};
\node (v3) at (-8,1) {$r_1$};

\node (v7) at (-8,-2) {$r_n$};

\draw  (v1) edge[->, line width=2] (v3);
\draw  (v1) edge[->, line width=2] (v4);
\draw  (v1) edge[->, line width=2] (v5);

\draw  (v1) edge[->, line width=2] (v7);

\draw  (v3) edge[->] (v8);
\draw  (v4) edge[->] (v8);
\draw  (v5) edge[->] (v8);

\draw  (v7) edge[->] (v8);
\draw[->, line width=2] (-4,-1) .. controls (-5,2) and (-6,2) .. (-8,2) .. controls (-10,2) and (-11,2) .. (-12,-1);

\draw[dotted]  (v5) edge (v7);
\node [outer sep=0,inner sep=0,minimum size=0] (v9) at (-8,-3) {};
\draw[dotted]  (v7) edge (v9);
\end{tikzpicture}
\caption{The counterexample. (Thick arrows stand for countably infinite parallel edges.)} \label{ellenpelda}
\end{figure}
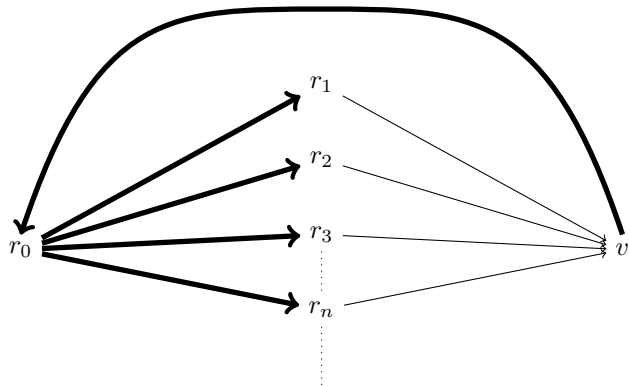     

\begin{conj}
Assume that  $ D=(V,A) $ is a digraph,  $ \kappa $ is an infinite cardinal and $ \mathcal{B}_i=(V_i,A_i)\ (i< \kappa) $ are
 edge-disjoint branchings in $ D $. Let  $D\bbslash \mathcal{B}=(V, A \setminus\cup_{i< \kappa}A_i) $. Suppose $ 
 D\bbslash \mathcal{B} $ does not contain 
 forward-infinite paths, and for all $ v\in V $ there is a system of edge-disjoint paths $ \left\lbrace P_i 
 \right\rbrace_{i<\kappa} $ in $ D\bbslash \mathcal{B} $ such that $ P_i $ goes from $ V_i $ to $ v $. Then the branchings 
 $ 
 \mathcal{B}_i $ can be extended to edge-disjoint spanning branchings of $ D $ without 
 changing their root sets. 
\end{conj}

\end{document}